\documentclass{article} 
\usepackage{amstext}
\usepackage{amsthm}

\usepackage{amssymb}
\usepackage{latexsym}
\usepackage{amsmath}
\usepackage{amscd}

\usepackage{tikz}
\usetikzlibrary{matrix,arrows}

\newtheorem{theorem}{Theorem}[section]

\newtheorem{lemma}[theorem]{Lemma}

\newtheorem{conjecture}[theorem]{Conjecture}

\newtheorem{problem}[theorem]{Problem}

\newcommand{\F}{\mathbb{F}}

\newcommand{\Z}{\mathbb{Z}}

\newcommand{\C}{\mathbb{C}}

\newcommand{\Pro}{\mathbb{P}}

\newcommand{\Lcal}{\mathcal{L}}

\newcommand{\Ocal}{\mathcal{O}}

\newcommand{\Div}{\mbox{\upshape{Div}}}

\newcommand{\disc}{\mbox{\upshape{disc}}}

\newcommand{\Mfrak}{\mathfrak{M}}

\title{Representation of powers by polynomials over function fields and a problem of Logic}

\author{Hector Pasten\\
Queen's University, Canada\\
\texttt{hpasten@gmail.com}}

\date{\today}

\begin{document}

\maketitle

\begin{abstract}  
We solve a generalization of B\"uchi's problem in any exponent for function fields, and briefly discuss some consequences on undecidability. This provides the first example where this problem is solved for rings of functions in the case of an exponent larger than $3$.
\end{abstract}

\tableofcontents


\section{Introduction and results}

Our starting point is the following conjecture by B\"uchi (see \cite{Mazur}), arising as an attempt to improve the negative answer to Hilbert's tenth problem given by Matiyasevic in 1970 after the work of J. Robinson, M. Davis and H. Putnam (see \cite{Matiyasevic})
\begin{conjecture}
There exists a constant $M$ with the following property. Suppose that $s_1,\ldots,s_M$ is a sequence of integer squares such that the second differences of the $s_i$ are constant and equal to $2$, that is
$$
s_{i+2}-2s_{i+1}+s_i=2,\quad i=1,\ldots,M-2.
$$ 
Then there is an integer $\nu\in\Z$ such that $s_i=(i+\nu)^2$ for $i=1,\ldots,M$.
\end{conjecture}
It is easy to see that if such $M$ does exist then $M\ge 5$, but no counterexample is known for $M=5$ and this conjecture is still an open problem.

Analogously, one can ask a similar question for other rings, higher order differences and higher powers, see \cite{PheidasVidaux1} for a detailed study of such extensions. A general statement for problems of this sort is rather complicated due to trivial exceptions arising in each particular case, so we will state here just the problem in the case of polynomials over the complex numbers

\begin{problem}\label{BuchiPol1} Let $n\ge 2$ be an integer. Is it true that there exists an integer $M=M(n)$ with the following property?

Given $q_1,q_2,\ldots,q_M\in\C[x]$, if the sequence of $n$-th powers of the $q_i$ has $n$-th differences constant and equal to $n!$, then either all the $q_i$ are constant, or there exists $\nu\in\C[x]$ such that $q_k^n=(k+\nu)^n$.
\end{problem}

This problem is of particular interest because a positive answer can be used to obtain consequences in logic in the spirit of the original motivation of B\"uchi. The reason for this is a celebrated theorem by Denef which establishes an analogue of the negative answer to Hilbert's Tenth Problem for polynomial rings in characteristic zero, see \cite{Denef1}.

Problem \ref{BuchiPol1} has been answered positively for $n=2$ (see \cite{Vojta} where Vojta actually proved an analogous statement for $n=2$ in a much more general context - function fields of curves in characteristic zero and meromorphic functions over $\C$) and $n=3$ (see \cite{PheidasVidaux3})\footnote{In personal communication, I have been informed that Hsiu-Lien Huang and Julie Tzu-Yueh Wang have recently solved B\"uchi's problem in the case of cubes for function fields. I deeply thank the authors for sending me their pre-print. We remark that the method of proof in the work of Huang and Wang is completely different from the method used here.}. In this work we answer positively to Problem \ref{BuchiPol1} and actually we prove a more general result for function fields. 

In order to state our main results, let us first make some remarks on Problem \ref{BuchiPol1}.

First of all, observe that if $u_1,\ldots,u_M$ is a sequence of elements in a (commutative unitary) ring $A$ whose sequence of $n$-th differences is $n!,n!,\ldots,n!$ then one has elements $a_0,a_1,\ldots,a_{n-1}\in A$ such that for $k=1,2,\ldots,M$
$$
u_k=k^n+a_{n-1}k^{n-1}+\ldots+a_1k+a_0,
$$
and clearly a sequence that admits such a representation has $n$-th differences $n!$. Thus we have the following equivalent statement for Problem \ref{BuchiPol1}

\begin{problem}\label{BuchiPol2} Let $n\ge 2$ be an integer. Is it true that there exists an integer $M=M(n)$ with the following property?

If $F(t)\in (\C[x])[t]$ is a monic polynomial of degree $n$ such that $F(\lambda)$ is an $n$-th power in $\C[x]$ for $\lambda=1,2,\ldots,M$, then either $F(t)$ has constant coefficients, or $F(t)=(t+\nu)^n$ for some $\nu\in\C[x]$.
\end{problem}

One can further generalize this problem by just requiring that $F(\lambda)$ is an $n$-th power for at least $M$ distinct values of $\lambda\in \C$, not necessarily $\lambda=1,\ldots,M$. Moreover, we can replace $\C$ by some other field $K$ of characteristic zero, and $\C[x]$ by some other $K$-algebra $R$ of arithmetic interest (for example, $R$ being the function field of a variety over $K$).

\begin{problem}\label{BuchiPol3} Let $n\ge 2$ be an integer. Is it true that there exists an integer $M=M(n)$ with the following property?

If $F(t)\in R[t]$ is a monic polynomial of degree $n$ such that $F(\lambda)$ is an $n$-th power in $R$ for at least $M$ values of $\lambda\in K$, then either $F(t)$ has constant coefficients (i.e. $F\in K[t]$), or $F(t)=(t+\nu)^n$ for some $\nu\in R$.
\end{problem}

Therefore a positive answer to Problem \ref{BuchiPol3} when $K=\C$ and $R=\C[x]$ would give a positive answer to Problem \ref{BuchiPol1}. 

In this last generalization we have insisted in requiring characteristic zero. The reason is that the problems we have discussed have negative answer in positive characteristic. For example over $\bar{\F}_p[x]$ for $p>2$ the polynomial 
$$
F(t)=\left(t+\frac{x^q+x}{2}\right)^2-\left(\frac{x^q-x}{2}\right)^2
$$ 
only represents squares as $t$ ranges in $\F_q\subseteq\bar{\F}_p$ for $q$ a power of $p$, but $F$ has non-constant coefficients and one can show that it is not of the form $(t+\nu)^2$. Nevertheless, one can also try to characterize such exceptions obtaining similar consequences in Logic, see for example \cite{PheidasVidaux2bis} and \cite{ShlapentokhVidaux}.


If $L$ is the function field of a curve over an algebraically closed field and $f\in L$ we say that $f$ is \textit{$k$-powerful} if all the zeros of $f$ have multiplicity at least $k$ (note that $k$ is not required to be attained and there is no assumption on the poles of $f$). Our main theorem is the following.

\begin{theorem}\label{MainTheorem}
Let $C$ be a smooth projective curve of genus $g$ over an algebraically closed field $K$ of characteristic zero. Let $n\ge 2$ be a positive integer and let  
$$
F(s,t)=s^n + a_{n-1}s^{n-1}t + \cdots + a_1st^{n-1} + a_0t^n\in K(C)[s,t]
$$ 
where at least one $a_i$ is non-constant. Assume that there is a set $B\subset \Pro^1(K)$ with at least
$$
M=M(n)= 2n(n+1)\left(g+n\binom{3n-1}{n}\right)
$$ 
elements and such that for each $b\in B$ we have that $F(b)$ is $\mu$-powerful in $K(C)$ for some fixed $\mu\ge n$. Then $\mu=n$ and $F$ is the $n$-th power of a linear polynomial in $K(C)[s,t]$. 
\end{theorem}

The statement has some obvious abuse of notation (when evaluating $F$ at points of $\Pro^1(K)$), which is harmless since the order of vanishing is well defined up to multiplication by non-zero scalars. 

As far as we know, this is the first case where the analogue of B\"uchi's problem in higher powers is solved completely for some ring of functions. Our techniques are completely different from the methods previously used to attack B\"uchi's problem in the case of functions. The previous methods were developed by Vojta in \cite{Vojta} and Pheidas and Vidaux in \cite{PheidasVidaux2} and \cite{PheidasVidaux3}. We believe that the extension of the methods of the previous authors for higher values of $n$ is not straightforward. Indeed, it was commented to me by Vidaux that his method works, in principle, for any \emph{given} $n$ as long as one is willing to work with systems of several differential equations, but making that method work for \emph{general} $n$ requires some new idea, or a systematic way to deal with such systems. 

Nevertheless, an extension of Vojta's method for higher values of $n$ would have remarkable arithmetic consequences, and on the other hand an extension of the method by Pheidas and Vidaux seems to be appropriate for the case of meromorphic functions over the complex numbers or non-archimedian fields. Indeed, in \cite{Transactions} we used both methods in order to explore arithmetic extensions of B\"uchi's original problem for number fields and to prove an analogue for $n=2$ in the case of $p$-adic meromorphic functions. 

Concerning partial results towards the solution of B\"uchi's problem for general $n$ (at least in the case of functions), in \cite{Proceedings} we considered an intermediate problem between the case $n=2$ and the case of general exponent for polynomial rings, this problem is called Hensley's problem. Then the results in \cite{Proceedings} were generalized for function fields in characteristic zero (see \cite{ShlapentokhVidaux}) and recently in the case of meromorphic functions over the complex numbers and non-archimedian fields (see \cite{chinas}). In all the cases the results were obtained by means of the Pheidas-Vidaux method mentioned above. Despite this progress, Hensley's problem for exponent $n$ implies B\"uchi's problem for exponent $n$ just in the case $n=2$, but for higher exponents Hensley's problem is a particular case of B\"uchi's problem.

The following slightly weaker form of Theorem \ref{MainTheorem} is more convenient for applications.  

\begin{theorem}\label{eMainTheorem}
Let $C$ be a smooth projective curve of genus $g$ over an algebraically closed field $K$ of characteristic zero, and let $n\ge 2$ be a positive integer. There exists a constant $N=N(n,g)$ depending only on $n$ and $g$ such that the following happens:

For all
$$
F(t)=t^n + a_{n-1}t^{n-1} + \cdots + a_1t + a_0\in K(C)[t],
$$ 
if $F(\lambda)$ is $n$-powerful in $K(C)$ for at least $N$ values of $\lambda\in K$ then either $F$ has constant coefficients or $F(t)=(t+\nu)^n$ for some $\nu\in K(C)$. 
\end{theorem}



Thus, we get as an immediate consequence a positive answer to Problem \ref{BuchiPol1} in the case of polynomials for example, because $n$-th powers in $K[x]$ are in particular $n$-powerful rational functions. 

As an application of Theorem \ref{eMainTheorem} one obtains the following consequences in Logic

\begin{theorem}\label{Logic} Let $\Lcal$ be the language $\{0,1,+,f_x,\alpha\}$ where $\alpha$ is a unary predicate. Let $\Mfrak$ be the $\Lcal$-structure with base set $\C[x]$ and where $f_x$ is interpreted as the map $u\mapsto xu$ and $\alpha$ is interpreted in one of the following ways:
\begin{enumerate}
\item $\alpha(u)$ means '$u$ is powerful'
\item $\alpha(u)$ means '$u$ is $k$-powerful' for fixed $k>1$
\item $\alpha(u)$ means '$u$ is a power'
\item $\alpha(u)$ means '$u$ is a $k$-th power' for fixed $k>1$.
\end{enumerate}
Then multiplication is positive existential $\Lcal$-definable over $\Mfrak$. In particular, the positive existential theory of $\Mfrak$ over $\Lcal$ is undecidable.
\end{theorem}

 We remark that Item 3 in Theorem \ref{Logic} has been recently proved by completely different methods by Garcia-Fritz as part of her MSc thesis (see \cite{Garcia}), and actually she managed to deal with even weaker languages and not only positive-existential theories. Also, Item 4 in the cases $k=2,3$ is already known after the work of Vojta, Pheidas, Shlapentockh and Vidaux (see \cite{Vojta}, \cite{ShlapentokhVidaux},\cite{PheidasVidaux2bis} and \cite{PheidasVidaux3}) also by different techniques. In all the mentioned cases, the strategy is to prove an arithmetic result in the spirit of Theorem \ref{eMainTheorem} and then use ideas of B\"uchi to obtain the results in Logic, see \cite{Survey} for a general exposition of these ideas, at least in the positive-existential case. The proof of Theorem \ref{Logic} from Theorem \ref{eMainTheorem} goes along the same lines of the work of the referred authors and we omit the details. Similar consequences for other structures (sub-rings of function fields of curves for example) over related languages are straightforward from Theorem \ref{eMainTheorem} as long as some version of Hilbert's Tenth Problem has been answered negatively for the corresponding structure. Such results can be obtained similarly, we let the details to the reader.


\section{Proof of Theorem \ref{MainTheorem}}

In this section we will use the notation introduced in the statement of Theorem \ref{MainTheorem}.

\subsection{A reduction}

We will need the following lemma.
\begin{lemma}\label{Linear} Let
$$
L=s+ct\in K(C)[s,t]
$$
with $c\in K(C)$ non-constant. There are at most 
$$
4+4g
$$ 
values of $b\in \Pro^1(K)$ for which $L(b)$ has only multiple zeros as a rational function on $C$ (after a choice of projective coordinates for $b$).  
\end{lemma}
\begin{proof}
Let $B'\subset\Pro^1(K)$ be the set of such $b$. Consider the map $\phi:C\to \Pro^1$ given by $p\mapsto [1:c(p)]$, and let $\check{\phi}:C\to \Pro^1$ be the composition of the dual map $\Pro^1\to\Pro^1$ with $\phi$, in coordinates this is $p\mapsto [-c(p),1]$. Let $d$ be the degree of $\check{\phi}$, then $c$ has $d$ poles counting multiplicity. If $b\in B'$ and $b\ne [1:0]$ then $b$ is a branch point of $\check{\phi}$, and since all the zeros of $L(b)$ are multiple we have $|\check{\phi}^{-1}(b)|\le d/2$. Therefore by Riemann-Hurwitz formula
$$
2-2g=2d-\sum_{q \mbox{ brach}} (d-|\check{\phi}^{-1}(q)|)\le 2d - \sum_{b\in B'-\{[1:0]\}}  \frac{d}{2}
\le 2d - (|B'|-1)\frac{d}{2}
$$
hence
$$
|B'|\le 5-\frac{4}{d}+\frac{4g}{d}\le 5+4g - \frac{4}{d}< 5+4g. 
$$
Since $|B'|$ is an integer we get $|B'|\le 4+4g$.
\end{proof}

We will prove the following lemma to reduce the proof of Theorem \ref{MainTheorem} to the proof of a simpler statement.
\begin{lemma}\label{Horizontal} It suffices to prove Theorem \ref{MainTheorem} under the additional hypothesis that $F$ cannot be factored as $F=GH$ for some $H\in K(C)[s,t]$ and some non-constant $G\in K[s,t]$.
\end{lemma}
\begin{proof} First we note that $F$ cannot be factored as $F=G(s,t)L(s,t)$ with $L$ linear on $s$, $t$ and $G\in K[s,t]$ because such an $F$ can be powerful for at most 
$$
4+4g+\deg G<4+4g+n<M(n)
$$ 
values of $[s:t]\in \Pro^1(K)$ (by Lemma \ref{Linear}). 

Suppose that the theorem is proved under the additional requirement, and given $F$ as in \ref{MainTheorem} suppose that we can factor it as $F=GH$ for some non-constant $G\in K[s,t]$ and some $H\in K(C)[s,t]$ and moreover assume that $G$ is the largest such factor. We can further suppose that $G,H$ are monic as polynomials on $s$ and that $H$ is not linear on $s,t$. Assume also that the hypothesis of Theorem \ref{MainTheorem} hold for $F$. Write
$$
H=s^{n'}+\cdots + b_1st^{n'-1}+ b_0t^{n'}\quad\mbox{ with } 2\le n'\le n
$$
and note that $G\in K[s,t]$ is homogeneous of degree $n-n'$. Since $G$ can vanish at most for $n-n'$ values of $[s:t]\in \Pro^1(K)$, we know that $H$ is $\mu$-powerful in $K(C)$ for at least 
$$
M(n)-(n-n')\ge M(n')=2n'(n'+1)\left(g+n'\binom{3n'-1}{n'}\right) 
$$ 
values of $b$ in $\Pro^1(K)$ with $\mu\ge n=n'+\deg_t(G)\ge n'$. 

Therefore, by maximality of $G$, we can apply to $H$ the version of the theorem that we are assuming as proved, so we must have $\mu=n'$ which implies $n=n'$ (because $\mu\ge n$) and $G$ is constant. 
\end{proof}

\subsection{Setup of the proof}

Let $S=C\times \Pro^1$ and take 
$$
F(s,t)=s^n + a_{n-1}s^{n-1}t + \cdots + a_1st^{n-1} + a_0t^n
$$
as in Theorem \ref{MainTheorem}. From now on, we call vertical (resp. horizontal) divisor on $S$ a divisor which is the pull-back of a divisor on $C$ (resp. on $\Pro^1$) by the corresponding projection. 

Let $D=(F)_0\in\Div(S)$ be the divisor of zeros of $F$ on $S$; this is nothing but the closure on $S$ of the divisor of zeros of $F$ on the generic fiber of the trivial family $S\to C$. Write 
$$
D=\sum_{i=1}^c m_iX_i
$$ 
where the $X_i$ are the reduced irreducible components of the support of $D$. Let 
$$
X=\bigcup_{i=1}^cX_i
$$
which is a reduced (but possibly reducible) curve on $S$. By Lemma \ref{Horizontal} we can assume that no $X_i$ is a horizontal divisor. Moreover, the rational normal curve in $\Pro^{n}$ is not contained in any proper linear subspace, in particular it is not contained in the dual of $[1:a_{n-1}(p):\cdots : a_0(p)]$ for any $p\in C$, therefore the $X_i$ cannot be vertical divisors.  

Let $\pi_1: S \to C$ and $\pi_2:S\to\Pro^1$ be the projection maps. Let $\nu_i:\tilde{X}_i\to X_i$ be the normalization of $X_i$, and define $h_i:\tilde{X}_i\to C$ and $f_i:\tilde{X}_i\to \Pro^1$ by $h_i=\pi_1\circ\nu_i$ and $f_i=\pi_2\circ\nu_i$. Note that $h_i$ and $f_i$ are non-constant morphisms because $X_i$ is not a vertical or horizontal divisor. Let $\epsilon_i$ and $\delta_i$ be the degrees of $h_i$ and $f_i$ respectively. We have 
$$
n=\sum_{i=1}^c m_i\epsilon_i
$$
so, in particular $\max\{m_i\}\le n$ with equality if and only if $n=m_1$ and $c=1$. We define
$$
d=\sum_{i=1}^c m_i\delta_i.
$$
Applying the Riemann-Hurwitz Formula to $h_i$ and $f_i$ we get
\begin{equation}\label{Zeuthen}
\epsilon_i\chi(C)-\sum_{p\in C}(\epsilon_i-|h_i^{-1}(p)|)=\delta_i\chi(\Pro^1)-\sum_{q\in \Pro^1}(\delta_i-|f_i^{-1}(q)|)
\end{equation}
where $\chi$ is the Euler characteristic. This formula is also known as Zeuthen Formula and the same idea works in general for irreducible correspondences on the product of two curves. In Equation \eqref{Zeuthen} we replace $\chi(\Pro^1)=2$, and then we add the equations with weight $m_i$ as $i$ ranges, to get

$$
n\chi(C)-\sum_{p\in C}\sum_{i=1}^c m_i(\epsilon_i-|h_i^{-1}(p)|)=2d-\sum_{q\in \Pro^1} \sum_{i=1}^c m_i (\delta_i-|f_i^{-1}(q)|)
$$

hence

\begin{equation}\label{WeightedZeuthen}
\sum_{q\in \Pro^1} \sum_{i=1}^c m_i (\delta_i-|f_i^{-1}(q)|)=\sum_{p\in C}\sum_{i=1}^c m_i(\epsilon_i-|h_i^{-1}(p)|) + 2d + 2n(g-1)
\end{equation}

For $p\in C$ define the set
$$
\Theta(p)=\{P\in X : \pi_1(P)=p\}
$$
then 
$$
|h_i^{-1}(p)|=\sum_{P\in \Theta(p)} |\nu_i^{-1}(P)|.
$$
Similarly, for $q\in \Pro^1$ let 
$$
\Gamma(q)=\{Q\in X : \pi_2(Q)=q\}
$$
and note that
$$
|f_i^{-1}(q)|=\sum_{Q\in \Gamma(q)} |\nu_i^{-1}(Q)|.
$$

We have
$$
\begin{aligned}
\sum_{q\in \Pro^1} \sum_{i=1}^c m_i (\delta_i-|f_i^{-1}(q)|)&\ge \sum_{q\in B} \sum_{i=1}^c m_i (\delta_i-|f_i^{-1}(q)|)\\
&=\sum_{q\in B} \sum_{i=1}^c m_i\delta_i-\sum_{q\in B} \sum_{i=1}^c m_i|f_i^{-1}(q)|\\
&= d|B|-\sum_{q\in B} \sum_{i=1}^c m_i\sum_{Q\in \Gamma(q)} |\nu_i^{-1}(Q)|
\end{aligned}
$$
and, if $E\subset C(K)$ is any finite set containing all the branch point of the $h_i$ then we get 
$$
\begin{aligned}
\sum_{p\in C}\sum_{i=1}^c m_i(\epsilon_i-|h_i^{-1}(p)|)&= \sum_{p\in E}\sum_{i=1}^c m_i(\epsilon_i-|h_i^{-1}(p)|)\\
&=\sum_{p\in E}\sum_{i=1}^c m_i\epsilon_i-\sum_{p\in E}\sum_{i=1}^c m_i|h_i^{-1}(p)|\\
&=n|E|-\sum_{p\in E}\sum_{i=1}^c m_i\sum_{P\in \Theta(p)} |\nu_i^{-1}(P)|.
\end{aligned}
$$
We will later choose a convenient $E$. After the above computation, Equation \eqref{WeightedZeuthen} implies

$$
d|B|-\sum_{q\in B} \sum_{i=1}^c m_i\sum_{Q\in \Gamma(q)} |\nu_i^{-1}(Q)|\le n|E|-\sum_{p\in E}\sum_{i=1}^c m_i\sum_{P\in \Theta(p)} |\nu_i^{-1}(P)| + 2d + 2n(g-1)
$$

that is

\begin{equation}\label{EqIntermedia}
d|B|-n|E|- 2d - 2n(g-1)\le\sum_{q\in B} \sum_{i=1}^c m_i\sum_{Q\in \Gamma(q)} |\nu_i^{-1}(Q)| -\sum_{p\in E}\sum_{i=1}^c m_i\sum_{P\in \Theta(p)} |\nu_i^{-1}(P)| 
\end{equation}

Let 
$$
Z=\{x\in X: X\mbox{ is singular at }x\}\cup (X\cap(F)_{\infty}) \subseteq S
$$ 
where $(F)_{\infty}$ stands for the divisor of poles of $F$ on $S$, which is nothing but the vertical divisor $C\times (F(q))_{\infty}$ for generic $q\in\Pro^1(K)$ (the choice of coordinates for $q$ does not affect this definition). Take $E$ as the union of $\pi_1(Z)$ and the set of all branch points of the maps $h_i$. Since $Z$ contains the singular points of $X$, $E$ is the same as the union of $\pi_1(Z)$ and the branch points of $\pi_1|_{X-Z}:X-Z\to C$.

Given $q\in\Pro^1(K)$, we note that $\Gamma(q)\setminus Z$ is included in the set 
$$
\{(p,q)\in S(K): F(q)\in K(C)\mbox{ vanishes at }p\}
$$ 
(fixing a choice of coordinates for $q$) because $Z\supseteq X\cap (F)_{\infty}$, but for $q\in B$ we know that $F(q)$ has multiplicity at least $\mu$ at each zero, hence for $q\in B$ we have
$$
\mu|\Gamma(q)\setminus Z|\le \deg (F(q))_0 \le (C\times q, D)=d
$$
where $(F(q))_0\in \Div(C)$ and $(\cdot,\cdot)$ denotes the intersection pairing on $\Div(S)$. So, from Inequality \eqref{EqIntermedia} we get
$$
\begin{aligned}
d|B|-n|E|- 2d - 2n(g-1)&\le \sum_{q\in B} \sum_{i=1}^c m_i\sum_{Q\in \Gamma(q)} |\nu_i^{-1}(Q)| -\sum_{p\in E}\sum_{i=1}^c m_i\sum_{P\in \Theta(p)} |\nu_i^{-1}(P)| \\
&\le \sum_{q\in B} \sum_{i=1}^c m_i\sum_{Q\in \Gamma(q)\setminus Z} |\nu_i^{-1}(Q)|\\
&= \sum_{q\in B} \sum_{Q\in \Gamma(q)\setminus Z} \sum_{i=1}^c m_i|\nu_i^{-1}(Q)|\\
&\le \max_i\{m_i\}\sum_{q\in B} \sum_{Q\in \Gamma(q)\setminus Z} \sum_{i=1}^c |\nu_i^{-1}(Q)|\\
&= \max_i\{m_i\}\sum_{q\in B} \sum_{Q\in \Gamma(q)\setminus Z} 1\\
&= \max_i\{m_i\}\sum_{q\in B} |\Gamma(q)\setminus Z|\\
&= \frac{\max_i\{m_i\}}{\mu}\sum_{q\in B} \mu|\Gamma(q)\setminus Z|\\
&\le  \max_i\{m_i\}|B|\frac{d}{\mu}.
\end{aligned}
$$
Note that we have used the fact that for $q\in B$ one has
$$
\sum_{Q\in \Gamma(q)\setminus Z} \sum_{i=1}^c |\nu_i^{-1}(Q)|=\sum_{Q\in \Gamma(q)\setminus Z} 1
$$
because for $Q$ a smooth point in $X$ there is one and only one $i$ such that $Q\in X_i$ (since the points where $X_i$ and $X_j$ meet are singular for $X$) and moreover for such $Q$ and $i$ one has $|\nu_i^{-1}(Q)|=1$ because $Q$ is a smooth point of $X_i$.

Write $m=\max_i\{m_i\}$. We get

\begin{equation}\label{EqMain1}
|B| \le |B|\frac{m}{\mu}+2\frac{n}{d}(g-1)+n\frac{|E|}{d}+ 2\le |B|\frac{m}{\mu}+2n\max\{0,g-1\}+n\frac{|E|}{d}+ 2
\end{equation}

Now we need an upper estimate for $|E|$. 


\subsection{Bounding $|E|$}

For convenience of the reader, we recall that
$$
F(s,t)=s^n + a_{n-1}s^{n-1}t + \cdots + a_1st^{n-1} + a_0t^n\in K(C)[s,t]
$$ 

A general horizontal divisor on $C\times\Pro^1$ meets $(F)_0$ in $d=\sum m_i\delta_i$ points counting multiplicity (by definition of $\delta_i$) hence $(F)_{\infty}$ is a formal sum of $d$ vertical lines counting multiplicity. So, we have that
\begin{itemize}
\item at most $d$ points in $C(K)$ are poles of some $a_i$, and 
\item each $a_i$ has at most $d$ poles counting multiplicity. 
\end{itemize}
Define
$$
U=\{p\in C :  a_i(p)\ne \infty,\forall i\}
$$ 
then $C-U$ consists of at most $d$ points.

Let $V=\Pro^{1}-\{[1:0]\}$. 

We will use the following well-known lemma. 

\begin{lemma}\label{Affine}
Let $Y$ be a smooth projective curve over $K$. Let $W$ be a non-empty proper open set in $Y$ obtained by removing the points $p_1,\ldots,p_r$. Then $W$ is an affine open set. In particular $U\times V$ is an affine open set of $S$.
\end{lemma}

The zero set of 
$$
\hat{F}=s^n+\cdots a_1s+a_0.
$$
agrees with $(F)_0$ on $C\times V$. We factor $\hat{F}$ as an element of $K(C)[s]$ in the following way

\begin{equation}\label{Factorization}
\hat{F}=\prod_{i=1}^R\hat{F}_i^{w_i}
\end{equation}

with $\hat{F}_i\in K(C)[s]$ distinct, non-constant on $s$, monic and irreducible. This is possible because $\hat{F}\in K(C)[s]$ is monic and non-constant on $s$. Define
$$
H=\prod_{i=1}^R\hat{F}_i.
$$

\begin{lemma}\label{FiReg}
The $\hat{F}_i$ have no poles on $U\times V$. 
\end{lemma}
\begin{proof}
If $\hat{F}_1$ has a pole through some point $(p,q)\in U\times V$ then it has a pole along $p\times V$. Thus some other $\hat{F}_i$ must vanish along $p\times V$ because $\hat{F}$ has no poles on $U\times V$, and this contradicts the fact that the $\hat{F}_i$ are monic and non-constant on $s$.
\end{proof}

\begin{lemma}\label{RedEq} We have that $(H)_0$ agrees with $(F)_0^{red}=\sum X_i$ on $U\times V$, that is, $H=0$ is an equation for $X$ on $U\times V$. 
\end{lemma}
\begin{proof}
First note that both $F$ and $H$ are regular on $U\times V$ (by Lemma \ref{FiReg}) hence their zero loci can be computed pointwise on $U\times V$. Given $P\in U\times V$ have that $P\in X$ if and only if $F(P)=0$ which happens if and only if $\prod_{i=1}^R\hat{F}_i^{w_i}(P)=0$.  The coefficients of $\prod_{i=1}^R\hat{F}_i^{w_i}$ are regular on $U$ so we conclude that for $P\in U\times V$ we have $P\in X$ if and only if $H(P)=0$. 

Now we prove that $(H)_0$ on $U\times V$ is reduced. Suppose it is not reduced, then a general vertical prime divisor on $U\times V$ meets $(H)_0$ in at least one multiple point, hence for general $p\in U$ we have that $\disc_p(H(s))=0$ where $\disc_p(H(s))$ stands for the discriminant of the polynomial $H((p,s))\in K[s]$ ($p\in U$ is given). This implies that $\Delta(p)=\disc_p(H(s))=0$ for general $p\in U$, where $\Delta\in K(C)$. So, $\Delta\in K(C)$ is the zero function and we conclude that $H$ has a multiple root as element of $K(C)[s]$. This contradicts the fact that $H\in K(C)[s]$ is a reduced polynomial in characteristic zero.
\end{proof}


Now, let $\Sigma$ be the set of points $P\in X\cap U\times V$ that are singular points or a ramification point for the projection $\pi_1|_X$. If $P_0=(s_0,p_0)\in \Sigma$ then 
$$
H(P_0)=0\mbox{ and }\partial_s H(P_0)=0
$$ 
hence $\Delta(p_0)=0$ where $\Delta\in K(C)$ is the discriminant of $H\in K(C)[s]$. We know that $\Delta\in K(C)$ is no the zero function because $H$ is a reduced polynomial in characteristic zero, and we also know that $\Delta$ is a polynomial expression on the coefficients of $H$. Let $v=\deg_s H\le n$. The intersection of a general horizontal divisor with $X$ has at most $d$ points counting multiplicity, hence the above lemma implies that each coefficient of $H$ has at most $d$ poles counting multiplicity.
The number of zeros of $\Delta$ on $U$ is at most the number of poles of $\Delta$ on $C$ counting multiplicities, that is at most
$$
\begin{aligned}
\#\left\{\stackrel{\mbox{monomials of }\Delta\mbox{ as a polynomial}}{\mbox{on the coefficients of }H} \right\}\cdot\left(\stackrel{\mbox{degree of }\Delta\mbox{ as a polynomial}}{\mbox{on the coefficients of }H}\right) \cdot d&\le \binom{3v-1}{v}(2v-2)d\\
&\le \binom{3n-1}{n}(2n-2)d 
\end{aligned}
$$
Finally, 
\begin{equation}\label{**}
|E|\le |\pi_1(\Sigma)|+ |C-U| + (C\times[1:0], X) \le \binom{3n-1}{n}(2n-2)d+2d.
\end{equation}


\subsection{A criterion for showing that $F$ is an $n$-th power}

The purpose of this section is to show that proving $\max\{m_i\}=n$ is enough to conclude that $F$ is an $n$-th power. This is done in Lemma \ref{clave} below.

We recall that
$$
D=(F)_0=\sum_{i=1}^c m_i X_i \in \Div(S)
$$
where the $X_i$ are the reduced irreducible components of the support of $D$, and
$$
\hat{F}=\prod_{i=1}^R\hat{F}_i^{w_i}
$$
as in Equation \eqref{Factorization}. 


\begin{lemma}\label{DivRed}
For each $i=1,\ldots,R$, the algebraic set in $U\times V$ defined by $\hat{F}_i=0$ is a (non-empty) reduced curve. Moreover 
$$
D=\sum_{i=1}^R w_i (\hat{F}_i)_0.
$$
on $U\times V$.
\end{lemma}

\begin{proof}
We have that 
$$
\sum_{i=1}^c X_i=(H)_0=\left(\prod\hat{F}_i\right)_0=\sum_{i=1}^R (\hat{F}_i)_0
$$ 
on $U\times V$, the first equality because of Lemma \ref{RedEq} and the last because on $U\times V$ the $\hat{F}_i$ are regular (by Lemma \ref{FiReg}). This shows that the curves $\hat{F}_i=0$ on $U\times V$ are reduced.  

Note also that each $\hat{F}_i$ must vanish somewhere in $U\times V$ because they are non-constant monic on $s$ having some non-constant coefficient (since the $X_i$ are not horizontal or vertical). 

Finally we have
$$
D|_{U\times V}=(\hat{F})_0=\left(\prod\hat{F}_i^{w_i}\right)_0=\sum_{i=1}^R w_i(\hat{F}_i)_0
$$ 
where the last equality is because the $\hat{F}_i$ are regular on $U\times V$.
\end{proof}

\begin{lemma}\label{PrimeDivisor}
We have that $D_i=(\hat{F}_i)_0$ is a prime divisor on $U\times V$ for each $i$. Moreover, $\Ocal_S(U\times V)=\Ocal_C(U)[s]$. 
\end{lemma}

\begin{proof}
By Lemma \ref{DivRed} $D_i$ is not the zero divisor and all of its coefficients are $1$ (we say that it is reduced). To show that $D_i$ is irreducible it is enough to show that $\hat{F}_i$ is a prime element in $A=\Ocal_S(U\times V)$ because $U\times V$ is affine by Lemma \ref{Affine}. 

First we note that $\Ocal_C(U)[s]\subseteq A$ but the canonical isomorphism 
$$
\Ocal_C(U)\otimes_K K[s]=\Ocal_C(U)\otimes_K \Ocal_{\Pro^1}(V)\longrightarrow A
$$ 
has image $\Ocal_C(U)[s]$ therefore $A=\Ocal_C(U)[s]$. 

\end{proof}


\begin{lemma}\label{Equality} We have that $R=c$ and $w_i=m_i$ for each $i$, up to reordering. 
\end{lemma}

\begin{proof}
By Lemma \ref{RedEq} on $U\times V$ we have
$$
\sum_{i=1}^c m_jX_j=D=\sum_{i=1}^R w_i(\hat{F}_i)_0
$$ 
where no $(\hat{F}_i)_0$ is the zero divisor, so, by Lemma \ref{PrimeDivisor} it suffices to show that $(\hat{F}_i)_0\ne (\hat{F}_j)_0$ for $i\ne j$ because those divisors are prime divisors. 
If $(\hat{F}_i)_0= (\hat{F}_j)_0$ for $i\ne j$ then we have $(\hat{F}_i)= (\hat{F}_j)$ because of Lemma \ref{FiReg}, so we get $\hat{F}_i=u\hat{F}_j$ for some $u\in \Ocal_S(U\times V)=\Ocal_C(U)[s]$ invertible, and in particular $u$ is constant on $s$. As all the $\hat{F}_i$ are monic this implies that $u$ is monic as a polynomial on $s$, therefore $u=1$. Hence $\hat{F}_i=\hat{F}_j$, a contradiction with the definition of the $\hat{F}_i$. 



\end{proof}


\begin{lemma}\label{clave}
We have that $F$ is the $n$-th power of a linear polynomial in $K(C)[s,t]$ if and only if $\max\{m_i\}=n$.
\end{lemma}
\begin{proof}
This follows by homogenizing the equation 
$$
\hat{F}=\prod_{i=1}^R\hat{F}_i^{w_i}
$$
with the variable $t$, Lemma \ref{Equality} and the fact that the $\hat{F}_i$ are not constant on $s$.
\end{proof}


\subsection{Completion of the proof}

From inequalities \eqref{EqMain1} and \eqref{**} we get
$$
\begin{aligned}
|B| &\le |B|\frac{m}{\mu}+2n\max\{0,g-1\}+n\frac{|E|}{d}+ 2\\
&\le |B|\frac{m}{\mu}+2n\max\{0,g-1\}+\frac{n}{d}\left(\binom{3n-1}{n}(2n-2)d+2d\right)+ 2\\
&= |B|\frac{m}{\mu}+2n\max\{0,g-1\}+\binom{3n-1}{n}(2n-2)n+2n+ 2\\
&< |B|\frac{m}{\mu}+2ng+2n^2\binom{3n-1}{n}
\end{aligned}
$$
that is

\begin{equation}\label{EqMain2}
|B|<|B|\frac{m}{\mu}+2ng+2n^2\binom{3n-1}{n}.
\end{equation}

Recall that $m=\max\{m_i\}\le n\le \mu$. We claim that $m=n=\mu$. Indeed, if $m<n$ then $m+1\le n\le \mu$ so \eqref{EqMain2} gives
$$
|B|<|B|\frac{n-1}{n}+2ng+2n^2\binom{3n-1}{n}
$$
hence
$$
|B|<2n^2g+2n^3\binom{3n-1}{n}.
$$
On the other hand, if $n<\mu$ then $m\le n\le \mu-1$, thus \eqref{EqMain2} gives
$$
|B|<|B|\frac{n}{n+1}+2ng+2n^2\binom{3n-1}{n}
$$
hence
$$
|B|<2n(n+1)g+2n^2(n+1)\binom{3n-1}{n}.
$$
In either case, we obtain
$$
2n(n+1)\left(g+n\binom{3n-1}{n}\right)= M\le |B|<2n(n+1)\left(g+n\binom{3n-1}{n}\right)
$$
a contradiction. This proves that $\max\{m_i\}=n=\mu$ and Theorem \ref{MainTheorem} follows from Lemma \ref{clave}.


\end{document}